\documentclass[12pt]{amsart}
\usepackage{amscd,amsmath,amsthm,amssymb,verbatim}

\numberwithin{equation}{section}
\newtheorem{theorem}{Theorem}[section]
\newtheorem{lemma}[theorem]{Lemma}

\newtheorem{corollary}[theorem]{Corollary}

\theoremstyle{definition}
\newtheorem{definition}[theorem]{Definition}
\newtheorem{remark}[theorem]{Remark}
\newtheorem{example}[theorem]{Example}

\begin{document}

\title{Generalized cover ideals and the persistence property}
\thanks{Version: \today}

\author{Ashwini Bhat}
\address{Department of Mathematics, 
Oklahoma State University, 
401 Mathematical Sciences,
Stillwater, OK 74078}
\email[A. Bhat]{ashwisb@ostatemail.okstate.edu}

\author{Jennifer Biermann}
\address{Department of Mathematics and Statistics, 451A Clapp Lab, 
Mount Holyoke College, South Hadley, MA 01075, USA}
\email[J. Biermann]{jbierman@mtholyoke.edu}

\author{Adam Van Tuyl}
\address{Department of Mathematical Sciences, 
Lakehead University, 
Thunder Bay, ON P7B 5E1, Canada}
\email[A. Van Tuyl]{avantuyl@lakeheadu.ca}
\urladdr[A. Van Tuyl]{http://flash.lakeheadu.ca/~avantuyl}

\keywords{monomial ideals, associated primes, trees, persistence property, index of stability
\\
\indent
2000 {\em Mathematics Subject Classification.} 13F20, 13A15, 05C25}

\begin{abstract}
Let $I$ be a square-free monomial ideal in $R = k[x_1,\ldots,x_n]$, and 
consider the sets of associated primes ${\rm Ass}(I^s)$  for all
integers $s \geq 1$.  Although it is known that the sets of associated 
primes of powers of $I$ eventually stabilize, there are few results about 
the power at which this stabilization occurs (known as the index of stability).  
We introduce a family of square-free monomial
ideals that can be associated to a finite simple graph $G$ that
generalizes the cover ideal construction. 
When $G$ is a tree, we explicitly determine ${\rm Ass}(I^s)$ for all $s \geq 1$.
As consequences, not only can we compute the index of stability,  
we can also show that this family of ideals has the persistence
property.
\end{abstract}

\maketitle

%%%%%%%%%%%%%%%%%%%%%%%%%%%%%%%%%%%%%%%%%%%%%%%%%%%%%%%%%%%%%%%%%%%%%%%%

\section{Introduction}

Let $I$ be an ideal of the polynomial ring $R = k[x_1,\ldots,x_n]$ with
$k$ a field.
A prime ideal $P \subseteq R$ is an {\it associated prime} of 
$I$ if there exists an element $T \in R$ such that $I:\langle T \rangle = P$.  
The {\it set of associated primes} of $I$,
denoted ${\rm Ass}(I)$, is the set of all prime ideals associated
to $I$.   We shall be interested in the sets ${\rm Ass}(I^s)$ as $s$ varies.
Brodmann \cite{B} proved 
that there exists an integer 
$s_0$ such that ${\rm Ass}(I^s) = {\rm Ass}(I^{s_0})$
for all integers $s \geq s_0$.  The least such integer $s_0$
is called the {\it index of stability}, and following \cite{HQ}, we denote
it by {\rm astab}$(I)$.   We are interested in the following problem which arises from Brodmann's
result:  determine {\rm astab}$(I)$ in terms
of the invariants of $R$ and $I$.  
Little is known about this problem, and in particular, there are few results 
providing exact calculations of {\rm astab}$(I)$. 
% The case of the index of 
%stability of an edge ideal was studied by Chen,  
%Morey and Sun in \cite{CMS} where they give an upper bound on {\rm astab}$(I)$.  

An upper bound on {\rm astab}$(I)$ for any monomial ideal $I$ was given by  Hoa \cite{H}.   
This bound is quite large and is in terms of the number of variables in the ring, 
the number of minimal generators of the ideal and the maximal degree of a minimal generator.  
%While not all ideals $I$ 
%satisfy the persistence property
%(see \cite{B}), any normal ideal $I$,
%that is, $I^s$ equals its integral closure $\overline{I^s}$ for all $s$,
%does satisfy this property (see \cite{R}).
Even when $I$ is a square-free monomial ideal, determining ${\rm astab}{(I)}$
remains a challenging problem.  A lower bound for ${\rm astab}{(I)}$ was given in \cite{FHVT}
in terms of the chromatic number of a hypergraph constructed from the
primary decomposition of $I$.  When $I$ is the edge ideal of a graph (a quadratic
square-free monomial ideal), Chen, Morey and Sung \cite{CMS}  provide an upper 
bound on {\rm astab}$(I)$.   However, the recent work of \cite{BHR,FHVT,FHVTconj,HQ,HRV,MMV} has 
suggested a possible answer.  In particular, 
Herzog and Qureshi \cite{HQ} posit that the bound ${\rm astab}(I) \leq 
\dim R-1 = n-1$ should hold for square-free monomial ideals $I$ (this
bound is significantly smaller than that given in \cite{H}).

Brodmann's results also suggest the following secondary question:  which ideals 
satisfy the {\it persistence property}, that is, for which ideals does the 
containment ${\rm Ass}(I^s)\subseteq {\rm Ass}(I^{s+1})$ hold for all $s \geq 1$?  
Recently, Kaiser, Stehl\'ik, and \u Skrekovski \cite{KSS} have shown that not all 
square-free monomial ideals have this property.  In light of this result, it is an interesting 
question to determine which square-free monomial ideals have the persistence property.  
Results in this direction have shown that the persistence property 
holds for many classes of square-free monomial ideals, including square-free principal Borel ideals 
\cite{A}, edge ideals \cite{MMV}, the cover
ideals of perfect graphs \cite{FHVT}, and polymatroidal ideals \cite{HRV}.

In this paper, we introduce a family 
of square-free monomial ideals (generalizing the notion of a cover ideal) that can be associated to 
a finite simple graph $G$, and study the associated
primes of their powers. More formally,
suppose that $G$ is a finite simple graph on the vertex set $V_G = \{x_1,x_2,\ldots,x_n\}$ with edge set $E_G$. 
For any $x \in V_G$, we let $N(x) = \{y ~|~ \{x,y\} \in E_G\}$
denote the set of {\it neighbours of $x$}.  By identifying the vertex $x_i$
with the variable $x_i$ in $R$, we define the 
following ideals.

\begin{definition}\label{partial}
Fix an integer $t \geq 1$.  The {\it partial $t$-cover ideal} of
$G$ is the monomial ideal
\[J_t(G) = \bigcap_{x \in V_G} \left(\bigcap_{\{x_{i_1},\ldots,x_{i_t}\} \subseteq
N(x)} \langle x,x_{i_1},\ldots,x_{i_t} \rangle \right).\]
\end{definition} 

\noindent
When $t=1$, our construction is simply the cover ideal of a finite
simple graph $G$ (see Section 2 for more details).  Recall that a graph
is a {\it tree} if it has no induced cycles.  Our main result 
is to show that for when $G$ is a tree, we can compute the index
of stability of $J_t(G)$, and show that this family
has the persistence property.

\begin{theorem}\label{maintheorem}
Let $G = (V_G,E_G)$  be a tree on $n$ vertices
and fix any integer $t \geq 1$.  Then the 
partial $t$-cover ideal $J_t(G)$ satisfies the persistence property.
Furthermore 
\[{\rm astab}(J_t(G)) = 
\left\{
\begin{array}{ll}
1 & \mbox{if $t=1$} \\
\min\{s ~|~ s(t-1) \geq \Delta(G) -1 \} & \mbox{if $t > 1$}
\end{array}
\right.
\]
where $\Delta(G)$ is the maximal degree of $G$, i.e., the largest degree of a vertex of $G$.
\end{theorem}
\noindent
In fact, we prove a stronger result (Theorem \ref{maintheoremtrees}) 
by determining the elements
of ${\rm Ass}(J_t(G)^s)$ for all $s \geq 1$.  Note that $\Delta(G) \leq n-1$,
so the upper bound suggested by Herzog and Qureshi also holds for this family.

Our paper is structured as follows.  In Section 2, we review the required
ingredients of associated primes and describe some of the properties
of $J_t(G)$.  In Section 3, we specialize to the case that $G = K_{1,n}$
is the star graph.  These graphs will play an important role
in our proof of Theorem \ref{maintheorem};  we also use these graphs
to answer a question raised in \cite{FHVT}.  Section 4 is devoted
to the proof of our main result.

\noindent
{\bf Acknowledgements.} 
Some of the results of Section 3 first appeared in \cite{Bhat}.
{\em Macaulay2} \cite{Mt} was used for computer experiments.
We thank T. H\`a and C. Francisco for their feedback.
The third author acknowledges the support of an NSERC Discovery Grant.

%%%%%%%%%%%%%%%%%%%%%%%%%%%%%%%%%%%%%%%%%%%%%%%%%%%%%%%%%%%%%%%%%%%%%%%%

\section{Preliminaries}

We continue to use the terminology and definitions introduced
in the previous section.   Throughout this paper, $\mathcal{G}(I)$ denotes
the unique set of minimal generators of a monomial ideal $I$.
For any $W = \{x_{i_1},\ldots,x_{i_s}\} \subseteq V_G$, we let
$x_W = x_{i_1}\cdots x_{i_s} \in R$.

 We first explain the 
significance of the name partial $t$-cover ideal in Definition
\ref{partial}. 
A {\it vertex cover}
of a graph $G$ is a subset $W \subseteq V_G$ which satisfies the
following property:  for any $x \in V_G$, either $x \in W$ or 
$N(x) \subseteq W$.  In other words, all the edges containing $x$ are covered.
We generalize this definition:  a {\it partial $t$-cover}
is a subset $W \subseteq V_G$ which satisfies the
following property: for any $x \in V_G$,
either $x \in W$ or there exists some subset $S \subseteq N(x)$
with $|S| = |N(x)|-t+1$ and $S \subseteq W$.  That is,
for each $x \in V_G$, all, but  perhaps $t-1$ of the edges 
containing $x$, are covered by $W$.   When $t =1$, this is simply
the definition of a vertex cover.  
The following lemma justifies our
choice of name for $J_t(G)$. 

\begin{lemma}  \label{genspartial}
Let $G = (V_G,E_G)$ be a finite simple graph
and $t \geq 1$ an integer.  Then
\[J_t(G) = \langle x_W ~|~ \mbox{$W \subseteq V_G$ is a partial $t$-cover} 
\rangle.\]
\end{lemma}

\begin{proof}
Let $m \in \mathcal{G}(J_t(G))$, and so 
$m = x_W$ for some $W \subseteq V_G$.  Suppose $W$ is not a partial
$t$-cover.  Then there exists a vertex $x$ such that $x \not\in W$,
and for all $S \subseteq N(x)$ with $|S| = |N(x)|-t+1$, there
is some $x_j \in S \setminus W$.  We claim that there are $t$
neighbours of $x$ not in $W$.  Let $S_1 = \{x_1,\ldots,x_{|N(x)|-t+1}\}$.
Because $W$ is not a partial $t$-cover, let $x_{i_1} \in S_1 \setminus W$.
Set $S_2 = (S_1 \setminus \{x_{i_1}\}) \cup \{x_{|N(x)|-t+2}\}$. 
Again, $W$ is not a partial $t$-cover, so there exists $x_{i_2} \in S_2 
\setminus W$.  We repeat $t$ times and find $t$ neighbours of $x$, say
$\{x_{i_1},\ldots,x_{i_t}\}$, that do not appear in $W$.  It then
follows that $m = x_W \not\in \langle x,x_{i_1},\ldots,x_{i_t} \rangle$ 
since none
of these variables appear in $x_W$.  But this contradicts the 
fact that $m \in J_t(G) \subseteq \langle 
x,x_{i_1},\ldots,x_{i_t} \rangle$.  Therefore $W$ is a partial $t$-cover.

For the converse, let $x_W$ be any square-free monomial which corresponds
to a partial $t$-cover.  Rewrite $J_t(G)$ as 
\footnotesize
\[J_t(G) = \left(\bigcap_{x \in W} \left(\bigcap_{\{x_{i_1},\ldots,x_{i_t}\} \subseteq
N(x)} \langle x,x_{i_1},\ldots,x_{i_t} \rangle \right)\right)
\cap
\left( \bigcap_{x \in V_G \setminus W} \left(\bigcap_{\{x_{i_1},\ldots,x_{i_t}\} \subseteq
N(x)} \langle x,x_{i_1},\ldots,x_{i_t} \rangle \right)\right)
.\]
\normalsize
If $x \in W$, then $x_W
\in \langle x,x_{i_1},\ldots,x_{i_t}\rangle$, so 
$x_W$ is in the first intersection.  If
$x \not\in W$, then there exists a subset $S \subseteq N(x)$
with $|N(x)|-t+1$ elements such that $S \subseteq W$.  But then
for any subset $T \subseteq N(x)$ with $|T| = t$, 
 $S \cap T \neq \emptyset$.  This implies that
$x_W \in \langle x,x_{i_1},\ldots,x_{i_t} \rangle$ for each
subset $\{x_{i_1},\ldots,x_{i_t}\}$ of $N(x)$ of size $t$.  
So $x_W$ is in the second intersection, thus completing
the proof.
\end{proof}

\begin{remark}
The Alexander dual (see \cite{MS} for the definition) 
of $J_t(G)$ is also of interest:
\[
I_t(G):= J_t(G)^\vee = \sum_{x \in V_G} \langle xx_{i_1}\cdots x_{i_t}
 ~|~ \{x_{i_1},\dots, x_{i_t}\}\subseteq N(x)\rangle.
\] 
If $t = 1$, then $I_1(G)$ is the edge ideal of $G$, and 
if $t=2$, then $I_2(G)$ 
 is the 2-path ideal of $G$ (see \cite{CD} for the definition).
The ideals $I_t(G)$ can be viewed as 
generalized edge ideals.  In a future paper, we will investigate
some of the properties of $I_t(G)$.
\end{remark}

We turn to the relevant results 
on associated primes of square-free monomial ideals.
Via the technique of localization, and using the fact
that localization and taking powers commute, we 
simply need to determine when the maximal ideal is an associated prime
of a monomial ideal.  The following lemma justifies this reduction.
The proof is similar to the proof of \cite[Lemma 2.11]{FHVT},
so is omitted.  Given a graph $G = (V_G,E_G)$ and subset $P \subseteq V_G$,
we write $G_P$ for the {\it induced graph} on $P$, i.e., the
graph with vertex set $P$, and edge set 
$E_{G_{P}} = \{e \in E_G ~|~ e \subseteq P\}$.

\begin{lemma}\label{localization}
Let $G$ be a graph on the 
vertex set $\{x_1, \dots, x_n\}$, and let $J_t(G)$ be the partial $t$-cover ideal of $G$.  The following are equivalent: 
\begin{enumerate}
\item[$(i)$]
$P= \langle x_{i_1}, \dots, x_{i_r}\rangle \in {\rm Ass}(J_t(G)^s)$
in $R = k[x_1,\ldots,x_n]$
\item[$(ii)$] 
$P=\langle x_{i_1}, \dots, x_{i_r}\rangle \in {\rm Ass} (J_t(G_P)^s)$ in
$R_P = k[x_{i_1},\ldots,x_{i_r}]$.
\end{enumerate}
\end{lemma} 

The next lemma shows $P \in {\rm Ass}(J_t(G)^s)$ gives a necessary condition
on the graph $G_P$.

\begin{lemma}\label{connected}
Let $G$ be a graph on the vertex set $\{x_1, \dots, x_n\}$, and 
let $J_t(G)$ be the partial $t$-cover ideal of $G$.
If $P =\langle x_{i_1}, \dots, x_{i_r}\rangle \in {\rm Ass}(J_t(G)^s)$,
then $G_P$ is connected.
\end{lemma}

\begin{proof}
By Lemma \ref{localization}, it is enough to show that if 
$\langle x_1,\ldots,x_n \rangle \in {\rm Ass}(J_t(G)^s)$ for some
$s$, then $G$ is connected.  Suppose $G$ is not connected, i.e.,
$G = G_1 \cup G_2$ with $G_1 \cap G_2 = \emptyset$.  After relabeling
the vertices, we can assume the vertices of $G_1$ are $\{y_1,\ldots,y_a\}$
and the vertices of $G_2$ are $\{z_1,\ldots,z_b\}$.  If $m 
\in \mathcal{G}(J_t(G))$, then $m = m_ym_z$ where $m_y$ is a square-free
monomial in the $y$ variables, and $m_z$ is a square-free monomial in the
$z$ variables, and furthermore, we must have $m_y \in \mathcal{G}(J_t(G_1))$,
and $m_z \in \mathcal{G}(J_t(G_2))$.  

Because $\langle x_1,\ldots,x_n \rangle = \langle y_1,\ldots,y_a,z_1,
\ldots,z_b \rangle$, and $\langle x_1,\ldots,x_n \rangle \in {\rm
Ass}(J_t(G)^s)$, there exists a monomial $T \not\in J_t(G)^s$ such that 
\begin{eqnarray*}
Ty_1 &= &m_1\cdots m_sM ~~\mbox{with $m_i \in \mathcal{G}(J_t(G))$} \\
& = & m_{y,1}m_{z,1} \cdots m_{y,s}m_{z,s}M_yM_z ~~ \mbox{with $m_i = m_{y,i}m_{z,i}$}
\end{eqnarray*}
where $m_{y,i} \in \mathcal{G}(J_t(G_1))$ and $m_{z,i} \in
\mathcal{G}(J_t(G_2))$, and $M_y$ (respectively $M_z$) is a monomial
in the $y$ variables (respectively the $z$ variables).  So, 
$T = (m_{z,1}\cdots m_{z,s}M_z)T'$
where $T'$ is a monomial in the $y$ variables.  But we also know that $Tz_1
\in J_t(G)^s$, so a similar argument allows us to write
$T = (u_{y,1}\cdots u_{y,s}U_y)T''$ where $T''$ is a monomial in
the $z$ variables, $U_y$ is a monomial in the $y$ variables,
and each $u_{y,j} \in \mathcal{G}(J_t(G_1))$.  
But this means
\begin{eqnarray*}
T  &= &(m_{z,1}\cdots m_{z,s}M_z)(u_{y,1}\cdots u_{y,s}U_y) 
=  (u_{y,1}m_{z,1})\cdots (u_{y,s}m_{z,s})U_yM_z.
\end{eqnarray*}
Now each $u_{y,i}m_{z,i} \in \mathcal{G}(J_t(G))$, so $T \in J_t(G)^s$,
a contradiction.   Thus $G$ is connected.
\end{proof}

Section 3 focuses on {\it star graphs}  $G = K_{1,n}$.  These are the graphs
with vertex set
$V_G = \{z,x_1,\ldots,x_n\}$ and edge set $E_G = \{\{z,x_i\} ~|~
1 \leq i \leq n\}$.  
The generators of $J_t(K_{1,n})$, 
as described by the next lemma,
follow directly from the definitions:

\begin{lemma} \label{generators}
Let $G = K_{1,n}$ with $V = \{z,x_1,\ldots,x_n\}$, 
and let $n \geq t \geq 1$.  Then
\[J_t(G) = \langle z \rangle + \langle x_{j_1}\cdots x_{j_{n-t+1}} ~|~
\{j_1,\ldots,j_{n-t+1}\} \subseteq \{1,\ldots,n\} \rangle.\]
\end{lemma}

The next example explains what we know about ${\rm Ass}(J_t(K_{1,n})^s)$ when
$t=1$;  the situation for $t \geq 2$ is explored in the next section.

\begin{example}\label{caset=1}
Let $G = K_{1,n}$ and $t=1$.  By Lemma \ref{generators},
 $J_1(G) = \langle z,x_1x_2\cdots x_n \rangle$.  But this
is a complete intersection, so
for all $s \geq 1$,
\[{\rm Ass}(J_1(G)^s) = {\rm Ass}(J_1(G))
= \{ \langle z,x_i \rangle ~|~ 1 \leq i \leq n\}. \] 
There are at least 
two ways to prove this result.  For
any complete intersection $J$,  $J^s = J^{(s)}$, the $s$-th
symbolic power of $J$ (see \cite{ZS}) and thus
${\rm Ass}(J^s) = {\rm Ass}(J)$ for all $s \geq 1$.  Alternatively, Gitler, Reyes,
and Villarreal have shown \cite[Corollary 2.6]{GRV}
that $J_1(G)$ is normal, i.e., 
$J_1(G)^s = \overline{J_1(G)^s}$, whenever $G$ is a bipartite graph, whence the
conclusion again follows.  
Because ${\rm astab}(J_1(G)) = 1$, $J_1(G)$ has the persistence property.
\end{example}

%%%%%%%%%%%%%%%%%%%%%%%%%%%%%%%%%%%%%%%%%%%%%%%%%%%%%%%%%%%%%%%%%%%%%%%%

\section{Star graphs}

Fix integers $n \geq t \geq 1$.  In this section we will
completely describe the sets ${\rm Ass}(J_t(G)^s)$ when $G = K_{1,n}$.  
We use our results to
give a new answer to a question raised by
Francisco, H\`a, and the third author in \cite{FHVT}.
Our main result is a corollary of the following theorem:

\begin{theorem}\label{maintheoremstar}
Fix integers $n \geq t \geq 1$ and let $G = K_{1,n}$
be the star graph on $V_G = \{z,x_1,\ldots,x_n\}$.  Set $J_t = 
J_t(G)$.  The following are equivalent:
\begin{enumerate} 
\item[$(i)$] $\langle z,x_1,\ldots,x_n \rangle \in{\rm Ass}(J_t^s)$ 
\item[$(ii)$] $s(t-1) \geq n-1$.
\end{enumerate}
\end{theorem}

We postpone the proof, but record its consequences:

\begin{corollary}\label{corstar}
Fix integers $n \geq t \geq 1$ and let $G = K_{1,n}$
be the star graph on $V_G = \{z,x_1,\ldots,x_n\}$.  For any
$s \geq 1$,
\[{\rm Ass}(J_t(G)^s) = \left. \left\{\langle z,x_{i_1},\ldots,x_{i_r}
\rangle  ~\right|~  t \leq r \leq \min\{n,s(t-1)+1\}\right \}.\]
Moreover,
\[{\rm astab}(J_t(G)) = 
\left\{
\begin{array}{ll}
1 & \mbox{if $t=1$} \\
\min\{s ~|~ s(t-1) \geq n -1 \} & \mbox{if $t > 1$.}
\end{array}
\right.\]
\end{corollary}

\begin{proof}
The result on ${\rm astab}(J_t(G))$ follows from the first statement.
Let $\mathcal{P}$ denote the set on the right hand side of the first statement.
Let $P \in {\rm Ass}(J_t(G)^s)$. 
Because $G_P$ is connected
by Lemma \ref{connected}, $P = \langle z,x_{i_1},\ldots,x_{i_r}
\rangle$, i.e., $P$ cannot be generated by a subset of $x$ variables.
Note that this means that $G_P = K_{1,r}$ for some $r$.  
Either $P$ is a minimal prime of $J_t(G)$, or contains a minimal
prime of $J_t(G)$, thus showing showing that
$t \leq r$.
By Lemma \ref{localization},  $\langle z,x_{i_1},\ldots,x_{i_r} \rangle
\in {\rm Ass}(J_t(G_P)^s)$, and so by
Theorem \ref{maintheoremstar},  $s(t-1) \geq r-1$,
i.e., $r \leq s(t-1) +1$.  Also, it is clear that $r \leq n$,
so $P \in \mathcal{P}$.

Conversely, suppose that $P = \langle z,x_{i_1},\ldots,x_{i_r} \rangle \in 
\mathcal{P}$.
Abusing notation, let $P \subseteq V_G$ denote the corresponding
vertices.  After localizing at $P$, $P \in {\rm Ass}(J_t(G_P)^s)$
by Theorem \ref{maintheoremstar} since $s(t-1) \geq r-1$.  
Lemma \ref{localization} then gives $P \in {\rm Ass}(J_t(G)^s)$.
\end{proof}

To prove Theorem \ref{maintheoremstar} we require
some information about our annihilator.

\begin{lemma}\label{annlemma}
Fix integers $n \geq t \geq 1$ and let $G = K_{1,n}$
be the star graph on $V_G = \{z,x_1,\ldots,x_n\}$.  Set $J_t = 
J_t(G)$.  Suppose that there exists a monomial
$T \in k[z,x_1,\ldots,x_n]$, $T \notin J_t^s$, such that
 $J_t^s:\langle T \rangle = \langle z,x_1,\ldots,x_n
\rangle$.  If $T = z^eT'$ where $z \nmid T'$, then
$T \mid z^e(x_1 \cdots x_n)^{s-e-1}.$
\end{lemma}

\begin{proof}
It suffices to prove that $T' | (x_1\cdots x_n)^{s-e-1}$.  Suppose that 
there exists some $x_i$ such that $x_i^{s-e} | T'$. Now 
$x_iT = z^ex_iT' \in J_t^s$, so 
\[z^ex_iT' = m_1m_2 \cdots m_sM ~~\mbox{with $M \in k[z,x_1,\ldots,x_n]$ and 
$m_i \in \mathcal{G}(J_t)$}.\]
We cannot have $x_i|M$.  If it did, then we could cancel $x_i$ from
both sides and have $T =z^eT' = m_1\cdots m_s(M/x_i) \in J_t^s$, which
contradicts the fact that $T \not\in J_t^s$. So, the variable $x_i$
appears at least $s-e+1$ times in $z^ex_iT'$, and thus, must appear
in at least $s-e+1$ of $m_1,\ldots,m_s$, because each $m_j$ is 
square-free.  In particular, we can assume $m_1 = x_im'_1$.
This means at most $e-1$ of $m_1,\ldots,m_s$ can be equal to $z$  
(no minimal generator of $J_t$ is divisible by both
$z$ and $x_i$ by Lemma \ref{generators}).  
So, $z$ must divide $M$, i.e., $M = zM'$.
So, to summarize,
\[z^ex_iT' = m_1m_2\cdots m_sM = (x_im'_1)m_2\cdots m_s(zM').\]
If we cancel $x_i$ from both sides, we get
\[T = z^eT' = (m'_1)m_2\cdots m_s(zM').\]
But $m_2,\ldots,m_s,z  \in \mathcal{G}(J_t)$, which means $T \in J_t^s$. 
This is our desired contradiction.
\end{proof} 
 
We are now ready to prove Theorem \ref{maintheoremstar}.

\begin{proof}
(of Theorem \ref{maintheoremstar})  Note that if $t=1$, then
Example \ref{caset=1} implies
$\langle z, x_1,\ldots,x_n \rangle \in {\rm Ass}(J_1(G)^s)$ if and only
if $n=1$ if and only if $0 = s(t-1) \geq n-1$.  So, we assume $t > 1$.

$(i) \Rightarrow (ii)$.
If $\langle z,x_1,\ldots,x_n \rangle \in {\rm Ass}(J_t^s)$, then 
there exists a monomial $T \notin J_t^s$
such that $J_t^s: \langle T \rangle = \langle z,x_1,\ldots,x_n\rangle$. 
Rewrite $T$ as $T = z^eT'$ where $z \nmid T'$.  We now claim that 
\begin{equation}\label{specialmonomial}
z^e(x_1\cdots x_{n-t+2})^{s-e}(x_{n-t+3}\cdots x_n)^{s-e-1} \in J_t^{s+1}.
\end{equation}
Indeed, by Lemma \ref{annlemma}, $z^eT' | z^e(x_1\cdots x_n)^{s-e-1}$.
Now $x_1z^eT' \in J_t^s$, which means 
$$z^ex_1^{s-e}(x_2\cdots x_n)^{s-e-1} \in J_t^s.$$
But $x_2\cdots x_{n-t+2} \in J_t$, so multiplying these two elements together
gives us the desired element in $J_t^{s+1}$.  

We proceed by a degree argument.  
By \eqref{specialmonomial}
there
exist generators $m_1, \ldots, m_{s+1}$ of $J_t$ such that 
\[z^e(x_1\cdots x_{n-t+2})^{s-e}(x_{n-t+3}\cdots x_n)^{s-e-1} =m_1 \cdots m_{s+1}M.\]
By Lemma \ref{generators},  $f$ of these 
generators are of the form $z$, and the 
remaining $s+1-f$ generators are of degree $n-t+1$ and 
have the form $x_{j_1}\cdots x_{j_{n-t+1}}$ for some
$\{j_1,\ldots,j_{n-t+1}\} \subseteq \{1,\ldots,n\}$.  Note that we must have 
$f \leq e$, and thus, looking at the degree of the generators in the
$x$ variables, we must have
\[
(s+1-f)(n-t+1) \leq (n-t+2)(s-e) + (t-2)(s-e-1) = (s-e)n - (t-2).\]
Expanding out the left hand side gives
\[sn-st+s+n-t+1-fn+ft-f \leq sn -en -t + 2.\]
Removing $sn$ and $-t$  from both sides and using the fact that $-en \leq -fn$ and $0 \leq f(t-1)$
gives $-st+s+n \leq 1$, which implies $s(t-1) \geq n-1$,
as desired.

$(ii) \Rightarrow (i)$  Let $s_0 = \min\{s ~|~ s(t-1) \geq n-1 \}$.  We first
show that $\langle z,x_1,\ldots,x_n \rangle \in {\rm Ass}(J_t^{s_0})$.  

We construct our annihilator as follows.  Write out the variables $x_1,\ldots,x_n$
as a repeating sequence, i.e.,
\begin{equation}\label{word}
x_1,x_2,\ldots,x_n,x_1,x_2,\ldots,x_n,x_1,x_2,\ldots,x_n,x_1,\ldots.
\end{equation}
Let $T$ be the product of the first $s_0(n-t+1)-1$ variables in this sequence, 
that is,
\[T = \underbrace{x_1x_2\cdots x_nx_1x_2 \cdots x_nx_1 \cdots x_j}_{s_0(n-t+1)-1}.\]
The monomial $T \not\in J_t^{s_0}$.  We can see this by a degree argument because
$J_t$ is generated by monomials in the $x$ variables of degree $n-t+1$.

We make the crucial observation that the index $j$ of the last variable in
$T$ has the property that 
$n-t+1 \leq j \leq n$.  To see this, note that after $n-t+1$ steps in the sequence
\eqref{word} we are at vertex $x_{n-t+1}$, after $2(n-t+1)$ steps in the sequence
\eqref{word}, we are at the vertex $x_{n-2(t-1)} = x_{n-2t+2}$, after $3(n-t+1)$
steps, we are at $x_{n-3t+3}$, ..., and finally, after $(s_0-1)(n-t+1)$
steps, we are at vertex $x_{n-(s_0-1)(t-1)} = x_{n-s_0t+s_0+t-1}$.  By our choice of
$s_0$, $-s_0t+s_0 \leq -n+1$, so $n-s_0t+s_0+t-1 \leq t$.  In fact, after $(s_0-1)$
steps of size $(n-t+1)$ in our sequence \eqref{word}, this is the first time
we arrive at an index $\leq t$. At the same time, by our choice of $s_0$, we
have $(s_0-1)(t-1) < n-1$, so we are at an index $\geq 1$.
 When constructing $T$, we go an additional $n-t$
steps in the sequence.  This means that we arrive at an index between $n-t+1$ and $n$.

We next show $J_t^{s_0}:\langle T \rangle = \langle z,x_1,\ldots,x_n \rangle$.
Now $zT \in J_t^{s_0}$.  To see this, note that $z$ is a minimal generator of $J_t$, and 
every $n-t+1$ consecutive
variables in \eqref{word} is also a generator of $J_t$.  Thus, the product of the first
$(s_0-1)(n-t+1)$ elements of \eqref{word}
is in $J_t^{s_0-1}$, and so $z \in J_t^{s_0}:\langle T \rangle $.

Now take $x_i$ with $i \in \{1,\ldots, n\}$.  To show $x_iT \in J_t^{s_0}$, take the first $s_0(n-t+1)-1$ variables in \eqref{word}, and insert $x_i$ after its first appearance, i.e.,
\[
x_1,x_2,\ldots,x_i,x_i,x_{i+1},\ldots,x_n,x_1,x_2,\ldots,x_n,x_1,x_2,\ldots,x_n,x_1,\ldots,x_j.
\]
Think of these variables as being placed around a circle.  Starting at the second $x_i$,
move around the circle, grouping $n-t+1$ variables together.  Because we
have $s_0(n-t+1)$ variables, we end up with $s_0$ groups.  Because the index
of $j$ is between $n-t+1$ and $n$, each group will consist of $n-t+1$ distinct
variables, and thus, by Lemma \ref{generators}, when we multiply each group of $n-t+1$ distinct variables together, we
have a generator of $J_t$.  But this means that $x_iT \in J_t^{s_0}$ since 
$x_iT$ is expressed as a product of $s_0$ generators.
Thus, $\langle z,x_0,\ldots,x_n \rangle \subseteq J_t^{s_0}:\langle T \rangle \subsetneq \langle 1 \rangle$,
which completes the proof for the case $s_0$.

Now suppose that $s > s_0$.  Let $e = s-s_0$ and let $T$ be as above.  We will
show that $J_t^{s}:\langle z^eT \rangle = \langle z,x_1,\ldots,x_n \rangle$.
By a degree argument $z^eT \not\in J_t^s$, but $z(z^eT) \in J_t^s$ because, as noted above, $T \in J_t^{s_0-1}$ and $z^{e+1}
\in J_t^{e+1}$.  Similarly, $x_iz^eT \in J_t^{s}$ because
$z^e \in J_t^e$, and as above, $x_iT \in J_t^{s_0}$.  Hence
$J_t^{s}:\langle z^eT \rangle = \langle z,x_1,\ldots,x_n \rangle$.
\end{proof}

\subsection{An application}
Corollary \ref{corstar} allows us to answer
a question raised by Francisco, H\`a, and the third author
\cite{FHVT}.  We first recall some terminology.

A {\it hypergraph} $\mathcal{H}$ is a pair of sets
$\mathcal{H} = (\mathcal{X},\mathcal{E})$ where 
$\mathcal{X} = \{x_1,\ldots,x_n\}$
and $\mathcal{E}$ is a collection of subsets 
$\{E_1,\ldots,E_t\}$ with each $E_i \subseteq \mathcal{X}$.  We call
$\mathcal{H}$ a {\it simple} hypergraph if $|E_i| \geq 2$ for all $i$,
and if $E_i \subseteq E_j$, then $i=j$.   (When each $|E_i| = 2$, then
$\mathcal{H}$ is a finite simple graph.)    As in the case of 
graphs, we say a subset $W \subseteq \mathcal{X}$ is a {\it vertex
cover} if $W \cap E \neq \emptyset$ for all $E \in \mathcal{E}$.  In 
a manner analogous to the cover ideal, we can define the cover ideal of $\mathcal{H}$:
\[J(\mathcal{H}) = \langle x_W ~|~ 
W = \{x_{i_1},\ldots,x_{i_t}\} \subseteq \mathcal{X} ~~\mbox{is a vertex cover}
\rangle.\] 
A {\it colouring} of $\mathcal{H}$ is an assignment of a colour to each
vertex of $\mathcal{X}$ so that no edge $E$ is mono-coloured, i.e.,
each edge must contain at least two vertices of different colours.
The {\it chromatic number} of $\mathcal{H}$, denoted $\chi(\mathcal{H})$,
is the least number of colours required to colour $\mathcal{H}$.  The
chromatic number provides a lower bound on the index of stability of $J(\mathcal{H})$.

\begin{theorem}[{\cite[Corollary 4.9]{FHVT}}] For any finite
simple hypergraph $\mathcal{H}$, 
\[\chi(\mathcal{H}) -1 \leq {\rm astab}(J(\mathcal{H})).\]
\end{theorem}

It was asked in \cite[Question 4.10]{FHVT} if for each
$m \geq 1$, there exists a hypergraph $\mathcal{H}_m$ with $\chi(\mathcal{H}_m)-1+m \leq
{\rm astab}(J(\mathcal{H}_m))$, that is, could the index of stability
be arbitrarily larger than the chromatic number.   Wolff \cite{W}
showed that this is the case, even if $\mathcal{H}$ is a finite simple
graph.  Wolff's family of graphs requires $5m-1$ vertices.    We can
use Corollary \ref{corstar} to give another answer to this question which only requires $m+3$ vertices.

\begin{theorem} Fix an $m \geq 1$, and let $\mathcal{H}_m = 
(\mathcal{X}_{m},\mathcal{E}_{m})$
where
$\mathcal{X}_m = \{z,x_1,\ldots,x_{m+2}\}$ and $\mathcal{E}_m = 
\{\{z,x_i,x_j\} ~|~ 1 \leq i < j \leq {m+2}\}$.  Then
\[\chi(\mathcal{H}_m)-1+m \leq {\rm astab}(J(\mathcal{H}_m)).\]
\end{theorem}

\begin{proof}
First, $\chi(\mathcal{H}_m) = 2$ because each $x_i$ can be assigned
the same colour, and $z$ can be given a different colour.  Note
that $J(\mathcal{H}_m) = J_2(K_{1,m+2})$.  By Corollary \ref{corstar}, 
${\rm astab}(J(\mathcal{H}_m)) = 
{\rm astab}(J_2(K_{1,m+2})) \geq  m+1 = (2-1)+m  = 
\chi(\mathcal{H}_m)-1+m$.    
\end{proof}

%%%%%%%%%%%%%%%%%%%%%%%%%%%%%%%%%%%%%%%%%%%%%%%%%%%%%%%%%%%%%%%%%%%%%%%%%

\section{Associated primes of Generalized cover ideals of trees}

In this section we completely determine the 
associated primes of the ideals $J_t(\Gamma)^s$ 
when $\Gamma$ is a {\it tree}, that is,  a graph with no induced cycles.    
Theorem \ref{maintheorem} will follow directly from this result.
We begin by stating the main theorem of this section:

\begin{theorem} \label{maintheoremtrees}
Fix an integer $t \geq 1$ and let $\Gamma$ be a tree
on $n$ vertices.
Then for all $s \geq 1$,
\[{\rm Ass}(J_t(\Gamma)^s) = 
\left. \left\{ P = \langle x_{i_0},x_{i_1},\ldots,x_{i_r} \rangle
~\right|~  \Gamma_P = K_{1,r} ~\mbox{with $
t \leq r \leq \min\{n,s(t-1)+1\}$}\right\}.\]
\end{theorem}

\noindent
In other words, a prime is associated to $J_t(\Gamma)^s$ if and only if
the corresponding induced subgraph in $\Gamma$ is a star of a particular
size.

We require the following lemma which can be found in
\cite[Proposition 4.1]{JK}).  This lemma will
gives us some insight into the generators of $J_t(\Gamma)$.

\begin{lemma}\label{specialvertex}
For any tree $\Gamma$, there exists a vertex $x$ such that all, but
possibly one, of its neighbours have degree $1$.
\end{lemma}

We fix some notation to be used throughout the remainder of this
paper.  Let $\Gamma$ be a tree, and let $x$
be the vertex of Lemma \ref{specialvertex} with neighbours $y_1,\ldots,y_d$.
We can assume that $\deg y_1 = \cdots = \deg y_{d-1} = 1$
and $\deg y_d \geq 1$.  Using this notation, we have:

\begin{lemma}  \label{structure}
Let $\Gamma$ be a tree with partial $t$-cover ideal $J_t(\Gamma)$.
If $m \in \mathcal{G}(J_t(\Gamma))$, then $m$ has one of the following forms:
\begin{enumerate}
\item[$(i)$]  $m = y_{i_1}\cdots y_{i_{d-t+1}}m'$
\item[$(ii)$] $m = xm'$
\item[$(iii)$] $m = xy_dm'$
\end{enumerate} 
where in each case, $m'$ is not divisible by any of the variables $y_1, \dots, y_d$, $x$.  
\end{lemma}

\begin{proof}
By Lemma \ref{genspartial}, the minimal generators of $J_t(\Gamma)$ correspond
to the minimal partial $t$-covers of $\Gamma$.   The result will follow if we look at
the corresponding statement for minimal partial $t$-covers of $\Gamma$.

Let $W$ be a minimal partial $t$-cover of $\Gamma$.  First, suppose that $x \not\in W$.
By definition, $W$ must contain a subset $S \subseteq N(x)$ of size
$|N(x)|-t+1 = d-t+1$.  Because $N(x) = \{y_1,\ldots,y_d\}$, let us say that
$S = \{y_{i_1},\ldots,y_{i_{d-t+1}}\}$.  It now suffices to show that $W \setminus S$
does not contain any other neighbours of $x$.  If $t=1$, then $S = N(x)$, so this is
clear.  So, suppose that $t \geq 2$, and suppose that there is some 
$y_j \in N(x)\cap (W \setminus S)$.  There are two cases to consider:  $j \neq d$
and $j = d$.

If $j \neq d$, then Lemma \ref{specialvertex} gives $\deg y_j = 1$.  We claim
that $(W \setminus \{y_j\})$ is also a partial $t$-cover of $\Gamma$, thus
contradicting the minimality of $W$.  Indeed,
take any vertex $z$ of $\Gamma$.  Because $y_j$ is only adjacent to $x$,
for any vertex $z \not\in \{y_j,x\}$,  either $z$ is in $(W \setminus \{y_j\}) \subseteq 
W$ or all but
perhaps $t-1$ of the neighbours of $z$ are in $(W \setminus \{y_j\}) \subseteq
W$.  
We know that $x \not\in W$,
but because $S \subseteq (W\setminus \{y_j\}) \subseteq W$, 
we know that all but perhaps $t-1$ of the neighbours
of $x$ are in $(W \setminus \{y_j\})$.  Finally, although $y_j \not\in W$,
all but perhaps $t-1 \geq 2-1 =1$ of its neighbours belong to $W$.  But since $y_j$ only has
the neighbour $x$, $(W\setminus \{y_j\})$ is also a partial $t$-cover.
If $j = d$, then we can simply repeat the above argument to show that 
$(W \setminus \{y_{i_1}\})$ (remove one the vertices of $S$, but keep $y_d$)
creates a smaller partial $t$-cover.

Now consider the case that $x \in W$.  It suffices to show that 
$\{y_1,\ldots,y_{d-1}\} \cap W = \emptyset$. Then we will have the form $(ii)$
if $y_d \not\in W$, and the form $(iii)$ if $y_d \in W$.  Suppose that 
$y_j \in \{y_1,\ldots,y_{d-1}\} \cap  W$.  We claim that $(W \setminus \{y_j\})$ would also be 
a partial $t$-cover.  By Lemma \ref{specialvertex}, $\deg y_j = 1$, and $y_j$ is
only adjacent to $x$.  As argued above, for any vertex $z \not\in \{y_j,x\}$,
either $z$ or all but perhaps $t-1$ of its neighbours will belong to $(W \setminus \{y_j\})$.
The vertex $x$ is in $(W \setminus \{y_j\})$, and as for $y_j$, although 
$y_j \not\in (W \setminus \{y_j\})$, the unique edge containing $y_j$ is covered by $x$.
So $(W \setminus \{y_j\})$ is a partial $t$-cover, contradicting the minimality of $W$.
\end{proof}

\begin{proof} (of Theorem \ref{maintheoremtrees})
Let $\mathcal{P}$ denote the set on the right. 
Lemma \ref{localization} and Corollary
\ref{corstar} imply that every induced star graph of $\Gamma$ of the appropriate size will contribute an associated
prime;  more precisely, we already have $\mathcal{P} \subseteq {\rm Ass}(J_t(\Gamma)^s)$.
It therefore suffices to show that if $P \in {\rm Ass}(J_t(\Gamma)^s)$, then
$\Gamma_P$ is a star graph.  Corollary \ref{corstar} and 
Lemma \ref{localization} then imply the condition on the
size of the star graph, thus showing $P \in \mathcal{P}$.

We let $J = J_t(\Gamma)$.  If $P \in \operatorname{Ass}(J^s)$, 
by Lemma \ref{localization}
we can assume that $\Gamma_P = \Gamma$
and by Lemma \ref{connected}, we can assume that $\Gamma$ is connected.  Because
$\Gamma$ is a tree, so is $\Gamma_P$.   
So, we can apply Lemma \ref{specialvertex}.  That is, we can assume that
there is a vertex $x$ with neighbours $y_1,\ldots,
y_d$ such that $\deg y_1 = \cdots =\deg y_{d-1} = 1$, and $\deg y_d \geq 1$
in $\Gamma_P$.  It suffices to show that $\deg y_d= 1$.  Since $\Gamma_P$ is connected,
this would mean $\Gamma_P = K_{1,d}$.

So, suppose $y_d$ has a neighbour, say $w \neq x$.
We thus have $P = \langle y_1,\ldots,y_d,x,w,\ldots \rangle$.  We now
want to build a contradiction from this information.

Since $P \in \operatorname{Ass}(J^s)$, there exists
a monomial $T\notin J^s$ such that $J^s:\langle T\rangle = P$.  Because
$w \in P$, 
\[Tw = m_1\cdots m_sM ~~\mbox{with $m_i \in \mathcal{G}(J)$}.\]
By Lemma \ref{structure}, a generator of $J$ has one of three forms.
Let's say that $a$ of $m_1,\ldots,m_s$ are of type $(i)$, $b$ of $m_1\ldots,m_s$ are of type $(ii)$,
and $c$ are of type $(iii)$.  We then
have
\[T = T'y_1^{e_1}y_2^{e_2}\cdots y_{d-1}^{e_{d-1}}y_d^{e_d+c}x^{b+c}\] 
where $e_1+\dots+e_d = (d-t+1)a$ and $a+b+c=s$.  
Without loss of generality we may assume that $e_1= \max\{e_1, \dots, e_{d-1}\}$.

We now consider $Ty_1$.  Since $y_1 \in P$, $Ty_1 \in J^s$, that is,
\[Ty_1 = u_1\cdots u_sU ~~\mbox{with $u_j \in \mathcal{G}(J)$}.\]
First, note that $y_1$ does not divide $U$, since if it did we would then have
$T =  u_1\cdots u_s(U/y_1) \in J^s$, a contradiction.
Since $Ty_1 =  T'y_1^{e_1+1}y_2^{e_2}\cdots y_{d-1}^{e_{d-1}}y_d^{e_d+c}x^{b+c}$,
this means that (at least) $e_1+1$ of the generators $u_1,\ldots,u_s$
are divisible by $y_1$.  We may assume that after reordering 
that these generators are $u_1,\ldots,u_{e_1+1}$.

We next observe that $x$ also does not divide $U$.  
To see why, suppose that $U = xU'$.
As noted above,  
$u_1 = y_1y_{i_2} \cdots y_{i_{d-t+1}}m$ for some monomial $m$ not divisible by $x$.  
Note that $(u_1x)/y_1 = xy_{i_2} \cdots y_{i_{d-t+1}}m$ will also be a non-minimal
generator of $J$.  This means that 
\begin{eqnarray*}
Ty_1 &= &u_1\cdots u_sU = 
(y_1y_{i_2} \cdots y_{i_{d-t+1}}m)u_2\cdots u_s(xU') \\
& = & (xy_{i_2} \cdots y_{i_{d-t+1}}m)u_2\cdots u_s(y_1U').
\end{eqnarray*}
If we now cancel $y_1$ from both sides, this implies that $T \in J^s$, a contradiction.  So
$x$ cannot divide $U$, and thus at least $b+c$ of $u_1,\ldots,u_s$ are divisible
by $x$.  By Lemma \ref{structure}, they cannot be among
$u_1,\ldots,u_{e_1+1}$ since these are all divisible by $y_1$.
Let us say that they are $u_{e_1+2},\ldots,u_{e_1+b+c+1}$. 
To summarize, we now have
\begin{eqnarray*}
Ty_1 &=&  \underbrace{u_1\cdots u_{e_1+1}}_{\mbox{all divisible by $y_1$}}\cdot
\underbrace{u_{e_1+2}\cdots u_{e_1+1+b+c}}_{\mbox{all divisible by $x$}}\cdots u_s
U.
\end{eqnarray*}

We finish the proof by counting the degrees of the variables $y_2, \ldots, y_d$ in $Ty_1$.  
There are two cases to consider: ({\it Case 1}) there is a generator among 
$u_1, \ldots, u_s$ of type $(iii)$; and ({\it Case 2}) there is no generator 
among  $u_1, \ldots, u_s$ of type $(iii)$.

{\it Case 1:}  Suppose there is some $u_j = xy_dm$.  Then $y_d$ must divide 
every generator among $u_1, \dots, u_s$ of type $(i)$.  To see why, 
suppose that there is some generator $u_r = y_1y_{i_2}\cdots y_{i_{d-t+1}}m'$ 
with $y_{i_\ell} \neq y_d$ for all $2\leq \ell \leq d-t+1$.
\begin{eqnarray*}
Ty_1 =  u_1\cdots u_r \cdots u_j \cdots u_sU
&=&  u_1 \cdots (y_1y_{i_2}\cdots y_{i_{d-t+1}}m')\cdots (xy_dm) \cdots u_sU\\
	&=&  u_1 \cdots (xm') \cdots (y_{i_2}\cdots y_{i_{d-t+1}}y_dm)\cdots u_s (y_1U).
\end{eqnarray*}
Note that $xm', y_{i_2}\cdots y_{i_{d-t+1}}y_dm \in\mathcal{G}(J)$.  If we cancel $y_1$ from both sides, we get $T \in J^s$, which is a contradiction.  

Similarly, suppose that there is some generator 
$u_r = y_{i_1}\dots y_{i_{d-t+1}}m'$ with $y_{i_\ell} \neq y_1, y_d$ 
for all $1\leq \ell \leq d$, and let 
$u_1 = y_1y_{k_2}\cdots y_{k_{d-t}}y_dm''$ (since $u_1$ is divisible by $y_1$, 
it must also be divisible by $y_d$ by above).  Since $y_{i_\ell} \neq y_1$ for
 all $1\leq \ell \leq d-t+1$, there is some variable among $y_{i_1}, \ldots, y_{i_{d-t+1}}$
 which does not divide $u_1$.  Without loss of generality, assume that 
$y_{i_1}$ does not divide $u_1$.  Then
\begin{eqnarray*}
Ty_1 &=&  u_1\cdots u_r \cdots u_j \cdots u_sU\\
	&=&   (y_1y_{k_2}\cdots y_{k_{d-t}}y_dm'') \cdots (xy_dm) 
\cdots (y_{i_1}y_{i_2}\cdots y_{i_{d-t+1}}m')\cdots u_sU\\
	&=&  (y_{i_1}y_{k_2}\cdots y_{k_{d-t}}y_dm'') \cdots 
(y_{i_2}\cdots y_{i_{d-t+1}}y_dm) \cdots (xm')\cdots u_s (y_1U).
\end{eqnarray*}
The monomials $xm'$, $y_{i_2}\cdots y_{i_{d-t+1}}y_dm$, and 
$y_{i_1}y_{k_2}\cdots y_{k_{d-t}}y_dm''$ are generators of $J$, so if we cancel $y_1$ from both sides this leads to the contradiction $T \in J^s$.  So if there is some $u_j = xy_dm$, then every generator of type $(i)$ among $u_1, \dots, u_s$ is divisible by $y_d$.  
  
Now consider the monomials $u_1, \ldots, u_{e_1+1}$.  After relabeling, 
we may assume that $y_1, \ldots, y_j$ and $y_d$ divide all of $u_1, \ldots, u_{e_1+1}$, 
and that each of the remaining variables $y_{j+1}, \ldots, y_{d-1}$ do not divide 
at least one of the generators $u_1, \dots, u_{e_1+1}$.  We now count the number 
of times that the variables $y_{j+1}, \ldots, y_{d-1}$ occur in the generators 
$u_1, \ldots, u_s$.  Each of $u_1, \ldots u_{e_1+1}$ are divisible by exactly 
$d-t+1$ of the variables $y_1, \ldots, y_d$ including $y_1, \ldots, y_j$ and $y_d$. 
 Therefore exactly $d-t+1-(j+1) = d-t-j$ of the variables $y_{j+1}, \ldots, y_{d-1}$ 
divide each of $u_1, \ldots, u_{e_1+1}$.  In addition, the variables $y_{j+1}, \ldots, y_{d-1}$
 may divide each of the monomials $u_{e_1+b+c+2}, \ldots, u_s$ (there
 are $s-(e_1+b+c+1) = a-e_1-1$ such monomials).   Since $y_d$ divides every 
generator of type $(i)$ in the list $u_1, \ldots, u_s$, at most $d-t$ of the 
variables $y_{j+1}, \ldots, y_{d-1}$ divide each of the generators $u_{e_1+b+c+2}, \ldots, u_s$. 
 In total, the number of times that the variables $y_{j+1}, \ldots, y_{d-1}$ divide
 the monomials $u_1, \ldots, u_s$ is at most
\[
(d-t-j)(e_1+1) + (d-t)(a-e_1-1) = (d-t)a - j e_1 -j.
\]
On the other hand, since $T = T'y_1^{e_1}\cdots y_{d-1}^{e_{d-1}}y_d^{e_d+c}x^{b+c}$, 
the number of times that the variables $y_{j+1}, \ldots, y_{d-1}$ divide $T$ is at least  
\begin{eqnarray*}
e_{j+1}+ \cdots + e_{d-1} &=&  e_1+\cdots+e_d - (e_1+\cdots +e_j+e_d)\\
	&=&  (d-t+1)a - (e_1+\cdots + e_j + e_d).
\end{eqnarray*}
Since $e_1 = \max\{e_1, e_2, \ldots, e_{d-1}\}$ we have $e_1+e_2+\cdots+e_j\leq je_1$.  So
\begin{eqnarray*}
(d-t+1)a - (e_1+\cdots + e_j + e_d) &\geq& (d-t+1)a - (je_1 + e_d)\\
	&=& (d-t+1)a -je_1-e_d.
\end{eqnarray*}
And since $e_d$ is the number of times that the variable $y_d$ appears among the (square-free) monomials $m_1, \dots, m_a$ we have $a\geq e_d$.  So
\begin{eqnarray*}
(d-t+1)a -je_1-e_d &\geq&(d-t+1)a - je_1 -a
	= (d-t)a -je_1.
\end{eqnarray*}
Since $j\geq 1$, this number is larger than the number of times that the variables $y_{j+1}, \ldots, y_{d-1}$ divide $u_1, \ldots , u_s$.  Therefore, there must be some $y_k$ with $j+1\leq k\leq d-1$ which divides $U$.  Let $U = y_kU'$.  By assumption, there is some monomial among $u_1, \ldots, u_{e_1+1}$ which is not divisible by $y_k$.  Without loss of generality, say $y_k \nmid u_1$.  Then $u_1 = y_1y_{i_2}\cdots y_{i_{d-t+1}}m'$ for some monomial $m'$ with $y_{i_\ell}\neq y_k$ for all $2\leq \ell \leq d-t+1$.  Then
\begin{eqnarray*}
Ty_1 = u_1\dots u_s U
	&=&  (y_1y_{i_2}\cdots y_{i_{d-t+1}}m')  u_2\cdots  u_s(y_kU')\\
	&=&  (y_ky_{i_2}\cdots y_{i_{d-t+1}}m') u_2 \cdots  u_s(y_1U').
\end{eqnarray*}
Since $y_ky_{i_2}\cdots y_{i_{d-t+1}}m' \in \mathcal{G}(J)$ this 
implies that $T \in J^s$, which is a contradiction.  
So $w \not\in P$, and thus $\Gamma_P = K_{1,d}$, as desired.

{\it Case 2:}   Suppose that no generator among $u_1, \dots, u_s$ is of the 
form $xy_dm'$ (which implies $c=0$).  
Assume again that each of the variables $y_1, \dots, y_j$ with $1 \leq j < d$ 
divides each of the monomials $u_1, \dots, u_{e_1+1}$ and that the variables 
$y_{j+1}, \ldots, y_{d-1}$ do not.  Note that $y_d$ may or may not divide
every monomial in $u_1,\ldots,u_{e+1}$.
We will count the variables $y_{j+1}, \dots, y_d$. 
 We saw in the previous case that we arrive at a contradiction if we assume that the 
variable $y_d$ divides every minimal generator of type $(i)$ in the list 
$u_1, \dots,  u_s$.   Therefore we may assume that there is some monomial 
of type $(i)$ among $u_1, \ldots , u_{e_1+1}, u_{e_1+b+2}, \dots, u_s$ which
 is of type $(i)$ and which is not divisible by $y_d$.  

Now $(d-t+1-j)$ of the 
variables $y_{j+1}, \dots, y_d$ divide each of the monomials $u_1, \dots, u_{e_1+1}$. 
 In addition, at most $(d-t+1)$ of the variables $y_{j+1}, \dots, y_d$ divide each of the 
monomials $u_{e_1+b+2}, \dots, u_s$.  In total the number of times that the 
variables $y_{j+1}, \dots, y_d$ divide the monomials $u_1, \dots, u_s$ is at most
\[
(d-t+1-j)(e_1+1)+(d-t+1)(s-(b+e_1+1)) = (d-t+1)a -je_1 -j.\\
\]
On the other hand, since $T = T'y_1^{e_1}\cdots y_{d-1}^{e_{d-1}}y_d^{e_d}x^{b}$ (because $c=0$ in this case), 
the number of times the variables $y_{j+1}, \dots, y_d$ divide $Ty_1$ is at least
\begin{eqnarray*}
e_{j+1}+\cdots+ e_d &=& e_1+\cdots+e_d - (e_1+\cdots +e_j)\\
	&=&  (d-t+1)a -(e_1+\cdots+e_j)\\
	&\geq&  (d-t+1)a -je_1
\end{eqnarray*}  
because $e_1 = \max\{e_1, \dots, e_{d-1}\}$.

Since $j\geq1$, this number is strictly greater than the number of times 
that $y_{j+1}, \ldots, y_d$ divide the monomials $u_1, \ldots, u_s$.  Therefore 
there is some $y_k$ with $j+1\leq k\leq d$ which divides $U$.  If $k \neq d$, 
then we know that there is some monomial among $u_1, \ldots, u_{e_1+1}$ which is not 
divisible by $y_k$.  Without loss of generality we may assume that $y_k$ does 
not divide $u_1 = y_1y_{i_2}\cdots y_{i_{d-t+1}}m'$.  Then
\begin{eqnarray*}
Ty_1 = u_1\cdots u_s U
	&=&  (y_1y_{i_2}\cdots y_{i_{d-t+1}}m')  u_2\cdots  u_s(y_kU')\\
	&=&  (y_ky_{i_2}\cdots y_{i_{d-t+1}}m') u_2 \cdots  u_s(y_1U').
\end{eqnarray*}
Since $y_ky_{i_2}\cdots y_{i_{d-t+1}}m' \in \mathcal{G}(J)$, 
this implies that $T \in J^s$ which is a contradiction.   

Finally, assume that none of $y_{j+1}, \dots, y_{d-1}$ divide $U$.  Then $y_d$ 
must divide $U$. Let $U = y_dU'$.  If there is some monomial among 
$u_1, \dots, u_{e_1+1}$ which is not divisible by $y_d$ then we arrive at a
 contradiction as above. If $y_d$ divides each of $u_1, \dots, u_{e_1+1}$,
 then there is some monomial in the list $u_{e_1+b+2}, \dots, u_s$ which is
 not divisible by $y_d$.  Without loss of generality, assume $u_s$ is not 
divisible by $y_d$.  So $u_s = y_{k_1}\cdots y_{k_{d-t+1}}m$, where 
$y_{k_\ell} \neq y_d$ for all $1 \leq \ell \leq d-t+1$ and 
$u_1 = y_1 y_{i_2}\cdots y_{i_{d-t}}y_dm'$.   Since $y_d$ divides $u_1$ and
 does not divide $u_s$ there is at least one of the variables $y_{k_1}, \ldots, y_{k_{d-t+1}}$ 
which does not divide $u_1$.  Assume that $y_{k_1}$ does not divide $u_1$.  Then
\begin{eqnarray*}
Ty_1 &=& u_1\cdots u_sU\\
	&=&  (y_1 y_{i_2}\cdots y_{i_{d-t}}y_dm')u_2\cdots u_{s-1}(y_{k_1}y_{k_2}\cdots y_{k_{d-t+1}}m)(y_dU')\\
	&=&(y_{k_1} y_{i_2}\cdots y_{i_{d-t}}y_dm')u_2\cdots u_{s-1}(y_{k_2}\cdots y_{k_{d-t+1}}y_dm)(y_1U').
\end{eqnarray*}
Since $y_{k_1} y_{i_2}\cdots y_{i_{d-t}}y_dm'$ and $y_{k_2}\cdots y_{k_{d-t+1}}y_dm$ 
are also minimal generators of $J$, this implies that $T$ is an element of $J^s$ 
which is a contradiction.

Therefore the associated prime $P$ cannot be of the form $P = \langle y_1, \dots, y_d, x, w, \ldots \rangle$.  In other words, $\deg (y_d) = 1$, so $\Gamma_P = K_{1, d}$ is a star graph as desired.
\end{proof}

We can now prove Theorem \ref{maintheorem}.

\begin{proof}(of Theorem \ref{maintheorem})  The persistence property is immediate
from our description of the sets ${\rm Ass}(J_t(\Gamma)^s)$ in Theorem \ref{maintheoremtrees}.
When $t=1$,  
${\rm astab}(J_1(\Gamma))=1$ since $\Gamma$ is bipartite.  So the result follows from \cite{GRV}.  When $t \geq 2$, let $x$
be a vertex with $\deg x = \Delta(\Gamma)$, i.e., a vertex of maximal degree.
Let $P = \{x\} \cup N(x)$.  Then $\Gamma_P = K_{1,\Delta(\Gamma)}$.
If we abuse notation, and let $P$ also denote the ideal generated by
the variables corresponding to the vertices in $P$, then $P \in {\rm Ass}(J_t(\Gamma)^s)$
if and only if  $s(t-1) \geq \Delta(\Gamma)-1$.  So ${\rm astab}(J_t(\Gamma)) \geq 
\min\{s ~|~ s(t-1) \geq \Delta(\Gamma)-1\}$. 

Let $s_0 = \min\{s ~|~ s(t-1) \geq \Delta(\Gamma)-1\}$ and suppose that 
  ${\rm astab}(J_t(\Gamma)) > s_0$.  Because $J_t(\Gamma)$ has the persistence property,
that means that there is a $P \in {\rm Ass}(J_t(\Gamma)^s) \setminus {\rm Ass}(J_t(\Gamma)^{s_0})$ 
with $s > s_0$.  We can assume $s$ is the smallest such integer 
with this property.
By Theorem \ref{maintheoremtrees},
$\Gamma_P = K_{1,r}$, and by Theorem \ref{maintheoremstar}, we must have $s(t-1) \geq r-1$.
Since $P \not\in {\rm Ass}(J_t(\Gamma)^{s_0})$, we must have $s_0(t-1) \not\geq r-1$.  But this means that $r > \Delta(\Gamma)$,
which implies that $\Gamma$ has a vertex of degree greater than $\Delta(\Gamma)$,
a contradiction.
\end{proof}

%%%%%%%%%%%%%%%%%%%%%%%%%%%%%%%%%%%%%%%%%%%%%%%%%%%%%%%%%%%%%%%%%%%%%%%%%

%%%%%%%%%%%%%%%%%%%%%%%%%%%%%%%%%%%%%%%%%%%%%%%%%%%%%%%%%%%%%%%%%%%%%%%%%


\begin{thebibliography}{99}
\bibitem{A} A. Aslam,
The stable set of associated prime ideals of a squarefree principal Borel ideal. Preprint (2013).
{\tt arXiv:1301.7152}

\bibitem{BHR} S. Bayati, J. Herzog, G. Rinaldo,
On the stable set of associated prime ideals of a monomial ideal.
Arch. Math. {\bf 98} (2012), 213--217.

\bibitem{Bhat} A. Bhat,
Associated Primes of Powers of the Alexander Dual of Path Ideals of Trees. 
MSc Thesis, Lakehead University (2012).
 
\bibitem{B} M. Brodmann, Asymptotic stability of ${\rm Ass}(M/I\sp{n}M)$.
Proc. Amer. Math. Soc. \textbf{74} (1979), 16--18.

\bibitem{CMS} J. Chen, S. Morey, A. Sung, The stable set of associated primes of the ideal of a graph.  Rocky Mountain J. Math {\bf 32}(2002), 71-89.


\bibitem{CD} A. Conca, E. De Negri,
M-Sequences, graph ideals and ladder ideals of linear type. 
J. Algebra \textbf{211} (1999), 599--624.


\bibitem{FHVT} C.A. Francisco, H.T. H\`a, and A. Van Tuyl, 
Colorings of hypergraphs, perfect graphs, and associated primes 
of powers of monomial ideals. 
J. Algebra \textbf{331} (2011), 224--242.

\bibitem{FHVTconj}
C.A. Francisco, H.T. H\`a, and A. Van Tuyl, 
A conjecture on critical graphs and connections to the 
persistence of associated
primes. 
Discrete Math. {\bf 310} (2010), 2176--2182. 


\bibitem{GRV}
I. Gitler, E. Reyes, R. Villarreal, 
Blowup algebras of ideals of vertex covers of bipartite graphs. 
{\it Algebraic structures and their representations}, 273–-279, 
Contemp. Math., {\bf 376}, Amer. Math. Soc., Providence, RI, 2005.

\bibitem{Mt} D.\ R.\ Grayson and M.\ E.\ Stillman,
Macaulay2, a software system for research in algebraic geometry.
{\tt http://www.math.uiuc.edu/Macaulay2/}.


\bibitem{HH} J. Herzog, T. Hibi, 
The depth of powers of an ideal. 
J. Algebra {\bf 291} (2005), 534--550. 

\bibitem{HQ} J. Herzog, A.A. Qureshi,
Persistence and stability properties of powers of ideals.
Preprint (2012). {\tt arXiv:1208.4684v2}

\bibitem{HRV}  J. Herzog, A. Riuf, M. Vladoiu,
The stable set of associated prime ideals of a polymatroidal ideal.
J. Alg. Combin. \textbf{37} (2013) 289--312. 

\bibitem{H} L. T. Hoa, Stability of associated primes of monomial ideals. 
Vietnam J. Math.  \textbf{34} (2006), 473-487.

\bibitem{JK} S. Jacques, M. Katzman,
Betti numbers of forests.
Preprint (2005). {\tt  arXiv:math/0501226v2}.  

\bibitem{KSS} T. Kaiser, M. Stehl\'ik, R. \u Skrekovski, Replication in critical graphs and the persistence of monomial ideals. 
To appear J.  Combin. Theory Ser. A (2013). {\tt arXiv:1301.6983v2}.  

\bibitem{MMV}
J. Mart\'inez-Bernal, S. Morey, R. Villarreal,
Associated primes of powers of edge ideals.
Collect. Math. {\bf 63} (2012), 361--374.

\bibitem{MS} E. Miller, B. Sturmfels, 
{\it Combinatorial Commutative Algebra.} 
GTM 227, Springer-Verlag, New York, 2004.

\bibitem{MV} S. Morey, R. Villarreal, 
Edge ideals: algebraic and combinatorial properties.
{\it Progress in Commutative Algebra, Combinatorics and Homology 1} 
(2012), 85--126.

\bibitem{R} L.J. Ratliff, Jr.  
On asymptotic prime divisors.
Pacific J. Math {\bf 111} (1984), 395--413.

\bibitem{W} S. Wolff,
Asymptotic growth of associated primes of certain graph ideals.
To appear Comm. Alg. (2012). {\tt arXiv:1204.3315v2}

\bibitem{ZS}
O.\ Zariski and P.\ Samuel,
{\em Commutative algebra. Vol. II.}
The University Series in Higher Mathematics, D. Van Nostrand Co., Inc., Princeton,
N. J.-Toronto-London-New York, 1960.


\end{thebibliography}
\end{document}